\documentclass[11pt, reqno]{amsart}
\usepackage{amsmath,amsfonts,amsthm,amssymb,latexsym}
\usepackage[a4paper,margin=3cm]{geometry}
\usepackage{graphicx}
\usepackage{url}
\usepackage[latin1]{inputenc}
\usepackage[T1]{fontenc}
\usepackage{graphicx}
\usepackage[frenchb,english]{babel}
\usepackage{calc}
\usepackage{amsmath,amsthm}
\usepackage{url}
\usepackage{times, amssymb, amscd, mathrsfs, graphicx, color}
\usepackage{enumerate}
\usepackage{hyperref}

\definecolor{darkred}{rgb}{0.9,0.1,0.1}

\newtheorem{theo}{Theorem}[section]

\newtheorem{coro}[theo]{Corollary}
\newtheorem{lem}[theo]{Lemma}

\newtheorem{Rq}[theo]{Remark}

\theoremstyle{plain}
\newtheorem*{hypo*}{Assumptions}

\newcommand{\un}{\mathbf{1}}
\newcommand{\N}{\mathbb{N}}                                              
\newcommand{\R}{\mathbb{R}}                                              
\newcommand{\C}{\mathbb{C}}
\newcommand{\p}{\mathbb{P}}
\newcommand{\E}{\mathbb{E}}                                            

\title[Simulation of quasi-stationary distributions]{A Stochastic Approximation Approach to Quasi-Stationary Distributions on Finite Spaces} %

\date{\today}

\author{ Michel~\textsc{Bena\"{i}m}} %
\address[Michel~\textsc{Bena\"{i}m}]{Institut de Math\'ematiques, Universit\'e de Neuch\^atel, Suisse} %
\email{\url{michel.benaim@unine.ch}}

\author{ Bertrand~\textsc{Cloez}}  %
\address[Bertrand~\textsc{Cloez}]{UMR INRA-SupAgro MISTEA, Montpellier, France}

\email{\url{Bertrand.Cloez@supagro.inra.fr}}%

\thanks{Partially supported by ANR-11-LABX-0040-CIMI within the program ANR-11-IDEX-0002-02}

\begin{document}

\maketitle

\begin{abstract}

This work is concerned with the analysis of a stochastic approximation algorithm for the simulation of quasi-stationary
 distributions on finite state spaces. This is a generalization of a method introduced by Aldous, Flannery and Palacios.
   It is shown that the asymptotic behavior of the empirical occupation measure of this process is precisely related to the asymptotic
   behavior of some deterministic dynamical system induced by a vector field on the unit simplex.
   This approach provides new proof of convergence as well as precise asymptotic rates for this type of algorithm.
   In the last part, our convergence results are compared with those of a particle system algorithm (a discrete-time version of the Fleming-Viot algorithm).
\end{abstract}

{\footnotesize \textbf{AMS 2010 Mathematical Subject Classification:} 65C20,  60B12, 60J10; secondary 34F05, 60J20 }

{\footnotesize \textbf{Keywords:} Quasi-stationary distributions - approximation method - reinforced random walks - random perturbations of dynamical
systems.}

\renewcommand*\abstractname{ \ }
\begin{abstract}
\tableofcontents
\end{abstract}


\section{Introduction}
Let $(Y_n)_{n\geq 0}$ be a Markov chain on a finite state space $F$ with transition matrix $P=(P_{i,j})_{i,j \in F}$. We assume that this process admits an (attainable) absorbing state, say $0$, and that $F^* = F\backslash \{ 0 \}$ is an irreducible class for $P$; this means that $P_{i,0}>0$ for some $i \in F^*$, $P_{0,i}=0$ for all $i\in F^*$ and $\sum_{k\geq 0} P^k_{i,j}>0$ for all $i,j\in F^*$. Note that there is no assumption here that  $P$ is aperiodic. For all $i\in F$ and any probability measure $\mu$ on $F$ (or $F^*$), we set
$$
\p_i \left( \ \cdot \ \right) = \p\left( \ \cdot \ | \ Y_0= i \right), \ \p_\mu =\sum_{i \in F} \mu(i) \p_i,
$$
and we let  $\E_i, \E_\mu$ denote the corresponding expectations. Classical results (i.e. \cite{DS65} and \cite[Theorem 2 p 53, Vol. 2]{G59}) imply that $Y_n$ is absorbed by $0$ in finite time and admits a unique probability measure $\nu$ on $F^*$, called a {\em quasi-stationary distribution} (QSD), satisfying, for every $k \in F^*$,
$$
 \nu(k) = \p_\nu ( Y_1 = k \ | \ Y_1 \neq 0) = \frac{\sum_{i \in F^*}\nu(i) P_{i,k}}{\sum_{i,j \in F^*} \nu(i) P_{i,j}}= \frac{\sum_{i \in F^*}\nu(i) P_{i,k}}{1-\sum_{i \in F^*} \nu(i) P_{i,0}}.
$$
If we furthermore assume that $P$ is aperiodic, then (see for instance \cite[Section 4]{DS65}) for any probability measure $\mu$ on $F^*$ and $k \in F^*$,
\begin{equation}
\label{eq:QLD}
\lim_{n \rightarrow + \infty} \p_\mu ( Y_n = k \ | \ Y_n \neq 0) =\nu(k).
\end{equation}
The existence and uniqueness of this measure can be proved through the Perron-Frobenius Theorem \cite[Theorem 2 p 53, Vol. 2]{G59} because  a probability measure $\nu$ is a QSD if and only if it is a left eigenvector of $P$ (associated to some eigenvalue $\lambda \in (0,1)$)\cite{DS65}; namely
\begin{equation}
\label{eq:vect-prop}
\nu P = \lambda \nu \Leftrightarrow \forall k \in F^*, \ \sum_{i\in F^*} \nu(i) P_{i,k} = \lambda \nu(k).
\end{equation}
Summing on $k$ the previous expressions gives the following expression of $\lambda$:
\begin{equation}
\label{eq:val-prop}
\lambda =  1 - \sum_{i\in F^*} \nu(i) P_{i,0}.
\end{equation}
Quasi-stationary distributions have many applications as illustrated for instance in \cite{CMM13,MV12,P,vDP13} and their computation is of prime importance.  This can be achieved with deterministic algorithms coming from numerical analysis \cite[section 6]{vDP13} based on equation \eqref{eq:vect-prop}, but these type of method fails to be efficient with  large state spaces. An alternative approach is to use stochastic algorithms  (even if naive Monte-Carlo methods are not well-suited as illustrated in the introduction of \cite{these-V}).
Our main purpose here is to analyze a class of such algorithms based on a method
 that was introduced by Aldous, Flannery and Palacios \cite{AFP} and which can be described as follows.

Let $\Delta$ be the unit simplex of probabilities over $F^*.$ For $x\in \Delta$, let $K[x]$ be a Markov kernel defined by
\begin{equation}
\label{eq:Kmu}
\forall i,j\in F^*, \ K[x]_{i,j} = P_{i,j} +  P_{i,0} x(j).
\end{equation}
and let  $(X_n)_{n\geq 0}$ be a process on $F^*$ such that for every $n\geq 0$,
\begin{equation}
\label{eq:transition1}
\forall i,j \in F^*, \ \p\left(X_{n+1} =j \ | \ \mathcal{F}_n  \right)= K[x_n]_{i,j}, \mbox{ on } \{X_n = i\},
\end{equation}
where
\begin{equation}
\label{eq:def-emp}
x_n = \frac{1}{n+1} \sum_{k=0}^{n} \delta_{X_k}
\end{equation}
stands for the {\em empirical occupation measure} of the process
and $\mathcal{F}_n= \sigma\{X_k, k\leq n\}.$
In words, the process behaves like $(Y_n)_{n\geq 0}$ until it dies (namely it hits $0$) and, when it dies, comes back to life in a state randomly chosen according to it's empirical occupation measure.

 This process is not Markovian
and can be understood as an urn process or a reinforced random walk.
Using the natural embedding of  urn processes into  continuous-time multi-type branching processes
\cite[section V.9]{AN04}, Aldous, Flannery and Palacios prove the convergence of $(x_n)$ to the QSD. As well illustrated in \cite{P07}, another powerful method for analyzing the behavior of processes with reinforcement is stochastic approximation theory \cite{BMP,D96} and its dynamical system counterpart \cite{B99}. Relying on this approach, we analyze a more general algorithm in which $(x_n)_{n\geq 0}$ is a weighted empirical measure. We then recover \cite[Theorem 3.8]{AFP} in this more general context with explicit rates of convergence. We also a provide a central limit theorem and prove the convergence of $(X_n)_{n\geq 0}.$ Note that,  when $\gamma_n=1/n$ the recent work \cite{BGZ} also provides a central limit theorem, using similar techniques.
This enables us to compare its convergence rates with a different algorithm which is a discrete-time version of the algorithm studied in \cite{BHM00,CT13, MM00, V11}
and is close (but different) to the one used in \cite{DM13, MG99}. We only give some qualitative bounds and we do not compare these two algorithms in terms of complexity.
  We describe it and give a new bound for the convergence based on \cite{BW03} in section \ref{sect:FV}.

\ \\

\textbf{Outline:} the next subsection introduces our main results. The proofs are in section \ref{sect:proof}. We study the dynamical system in \ref{sect:flow}, relate its long term behavior to the long term behavior of $(x_n)_{n\geq 0}$ in \ref{sect:synthese}, and end the proof in \ref{sect:main-proof}. Finally, Section \ref{sect:FV} treats the second algorithm based on a particle system.

\subsection{Main results}

Assume that $F^*$ contains $d\geq 2$ elements and let us define the unit simplex of  probability measures on $F^*$ by $\Delta = \left\{ x \in \mathbb{R}^d \  \Bigr\rvert  \, x_i \geq 0, \,  \sum_{i=1}^d x_i= 1 \right\}.$  We endow $\mathbb{R}^d$  with the classical $l^1$-norm:
$\Vert x \Vert  = \sum_{i \in F^*} |x(i)|$ and $\Delta$ with the induced distance
 (which corresponds, up to a constant, to the total variation distance). Given $x\in \Delta$, we denote by $\pi(x)$ the invariant distribution of $K[x]$, defined in \eqref{eq:Kmu}, and we let $h : \Delta \to T\Delta =  \left \{x   \in \mathbb{R}^d   \Bigr\rvert \sum_{i=1}^d x_i= 0 \right\}$ denote the vector field given by
 $h(x) = \pi(x)-x$. Our aim is to study the \textit{weighted empirical occupation measure} $(x_n)_{n\geq 0}$, defined for every $n\geq 0$ by
\begin{equation}
\label{eq:general-xn}
x_{n+1} =  (1-\gamma_n) x_n + \gamma_n \delta_{X_{n+1}} = x_n + \gamma_n (h(x_n) + \epsilon_n),
\end{equation}
where $\epsilon_n= \delta_{X_{n+1}} - \pi(x_n),$  $(\gamma_n)_{n\geq 0}$ is a decreasing sequence on $(0,1)$ verifying
\begin{equation}
\label{eq:hyp-gamma}
\sum_{n\geq 0} \gamma_n= + \infty \ \text{ and } \ \lim_{n \rightarrow + \infty} \gamma_n \ln(n)=0,
\end{equation}
and the  process $(X_n)_{n\geq0}$  satisfies \eqref{eq:transition1}. Let us set $\tau_0=0$,
\begin{equation}
\label{eq:limsup-gamma}
\tau_n= \sum_{k=1}^n \gamma_k, \text{ and } \ l(\gamma) = \limsup_{n\rightarrow + \infty} \frac{\ln(\gamma_n)}{\tau_n}.
\end{equation}
For instance, if
$$\gamma_n= A n^{-\alpha} \ln(n)^{-\beta}, A > 0, \alpha, \beta \geq 0,$$
then $$ l(\gamma) =  \left\{
      \begin{aligned}
        & 0, \ \text{ if } (\alpha,\beta) \in (0,1) \times \R_+,  \\
        &-1/A \ \text{ if } \alpha=1, \beta=0,  \\
        &- \infty \ \text{ if }  (\alpha,\beta) \in \{1\} \times (0,1].
      \end{aligned}
    \right.
$$
\begin{Rq}
\label{remweight}
The sequence  \eqref{eq:def-emp} corresponds to the choice $\gamma_n = \frac{1}{n+2}$.
More generally, let $(\omega_n)_{n\geq 0}$ be a sequence of positive number, if
$$
\gamma_n = \frac{\omega_n}{\sum_{i=0}^n \omega_k} \Leftrightarrow \omega_n =\frac{\kappa \gamma_n}{\prod_{k=0}^n (1-\gamma_i) },
$$
for some $\kappa >0$, then
$$
x_n = \frac{\sum_{i=0}^n \omega_i \delta_{X_i}}{\sum_{i=0}^n \omega_i}.
$$
Notice that with $\omega_n = n^{a}$ for $a > -1, \gamma_n\sim \frac{1+a}{n}.$
\end{Rq}

The sequence $(x_n)_{n\geq0}$ is often called a stochastic approximation algorithm with decreasing step \cite{B99,BMP,D96}. Its long time behavior can be related to the long time behavior of the flow $\Phi$ induced by $h$; namely the solution to
\begin{equation}
\label{eq:flow-principal}
  \left\{
      \begin{aligned}
        \forall t\geq 0, \forall x\in \Delta &, \ \partial_t \Phi(t,x) = h(\Phi(t,x)), \\
       \ \Phi(0,x)=x.
      \end{aligned}
    \right.
\end{equation}
 In order to state our main result, let us  introduce some notation. By Perron-Frobenius Theorem \cite[Theorem 2 p 53, Vol. 2]{G59}, eigenvalues of $P$ can be ordered as
$$
 1 > \lambda_1 \geq |\lambda_2| \geq \dots \geq |\lambda_d| \geq 0,
$$
where $\lambda_1= \lambda$ (defined in \eqref{eq:val-prop}) and $\lambda_i \neq \lambda$ for all $i \geq 2$. Set
\begin{equation}
\label{eq:R}
R= 1 - (1-\lambda) \max_{i\geq 2} \textsc{Re} \left(\frac{1}{1-\lambda_i}\right),
\end{equation}
where $\textsc{Re}$ is the real part application on $\C$. Since $|1-\lambda|<|1-\lambda_i|$ for   $i\geq 2$, we have $R>0$.

\begin{theo}[Convergence of $(x_n)_{n\geq 0}$ to the quasi-stationary distribution]
\label{th:main}
With probability one, $x_n$ tends to $\nu$. If furthermore $l(\gamma)<0$, then
$$
\limsup_{n\rightarrow + \infty} \frac{1}{\tau_n} \ln\left(\Vert x_n - \nu \Vert\right) \leq \max\left(-R, \frac{l(\gamma)}{2} \right) \ \text{ a.s.}
$$
\end{theo}

This leads to the following result which generalize and precise  the rates of convergence  of \cite[Theorem 3.8]{AFP}
\begin{coro}
\label{cor:Aldous}
Suppose $\gamma_n = \frac{A}{n}$ for some $A > 0$ (or, with the notation of remark \ref{remweight}, $\omega_n = n^{A-1}$)
then for all $\theta < \min\left(R A, 1/2\right)$, there exists a random constant $C>0$ such that
$$
\forall n \geq 0, \ \Vert x_n - \nu \Vert \leq C n^{-\theta} \ \text{a.s.}
$$
\end{coro}

Using general results on stochastic approximation, we are also able to quantify more precisely this convergence, as shown by the following theorem.

\begin{theo}[Central limit theorem]
\label{th:tcl}
If one of the following conditions is satisfied
\begin{enumerate}[i)]
\item $\sum_{k\geq 0} \gamma_k = + \infty, \sum_{k\geq 0} \gamma_k^2 < \infty$ and $\lim_{k \rightarrow + \infty} \gamma_{k}^{-1} \ln(\gamma_{k-1}/\gamma_k)=0$;
\item $\sum_{k\geq 0} \gamma_k = + \infty, \sum_{k\geq 0} \gamma_k^2 < \infty$ and $\lim_{k \rightarrow + \infty} \gamma_{k}^{-1} \ln(\gamma_{k-1}/\gamma_k)=\gamma_*^{-1} < 2R$;
\end{enumerate}
then there exists a covariance matrix $V$ such that
$$
\gamma_n^{-1/2} (x_n - \nu) \underset{n \to +\infty}{\overset{d}{\longrightarrow}} \mathcal{N}(0, V).
$$
\end{theo}

From Ces\`{a}ro Theorem, if Assumption \textit{ii)} of the last theorem holds then $\gamma_*^{-1} = -l(\gamma)$. In particular, under this assumption, the limiting result of Theorem \ref{th:main} is an equality. In case $\gamma_n =1/n$  this convergence result has been already proved when in the recent work \cite{BGZ} with a similar approach. Furthermore, this gives the following trivial consequence.

\begin{coro}[$L^p-$bound for the convergence of $(x_n)_{n\geq 0}$]
\label{cor:tcl}
Under the previous assumptions, there exists for all $p \geq 1$ $C_p > 0$ such that for every $n\geq 0$,
$$
\lim_{n \rightarrow \infty}  \gamma_n^{-1/2}\E\left[ \sum_{i\in F^*} | x_n(i) - \nu(i) |^p \right]^{1/p} = C_p
$$
\end{coro}

Note that this result extends \cite[Theorem 1.2]{DM06} and \cite[Theorem 2.2]{DM04} when considering a finite state space and this particular type of kernel $K[x]$.

Finally, not only the (weighted) empirical occupation measure of $(X_n)_{n\geq 0}$ converges almost surely to $\nu$ but $(X_n)$ itself converges in distribution to $\nu$ as shown by the next result.

\begin{coro}[Convergence in law to $\nu$]
\label{cor:Xn}
Let $(\mu_n)_{n\geq 0}$ be the sequence of laws of $(X_n)_{n\geq 0}.$ Then
$$
\lim_{n \rightarrow + \infty} \Vert \mu_n - \nu \Vert=0.
$$
If we furthermore assume that the assumptions of Theorem \ref{th:tcl} hold, there exists  $C > 0$ and $0 < \rho < 1$ such that
$$
\Vert \mu_{n+p} - \nu \Vert \leq C (\rho^p +   p \sqrt{\gamma_n}).
$$
\end{coro}


Proofs of these results are given in section \ref{sect:proof} and in particular in \ref{sect:synthese}.

\section{Study of the flows and proofs of our main results}
\label{sect:proof}
As explained in the introduction, the proof is based on the ODE method. We study $\Phi$ and apply its properties to $(x_n)_{n\geq 0}$ with classical results on perturbed ODE. So we decompose this section into three subsections: the study of the flow $\Phi$, the study of the noise $(\epsilon_n)_{n\geq 0}$ and finally the proof of the main theorems.

\subsection{Analysis of the flow}
\label{sect:flow}
For any $x, y\in \R^d$, we will use the following notation:
$$
\langle x,y \rangle = \sum_{i \in F^*} x(i)y(i),
$$
and $\mathbf{1}$ will denote the unit vector; namely $\mathbf{1}(i)=1$ for every $i\in F^*$. Let us begin by giving a more tractable expression for $\pi$. As $\hat{P}=(P_{i,j})_{i,j\in F^*}$ is sub-stochastic, the matrix $A=\sum_{k\geq 0} \hat{P}^k$ is well defined and is the inverse of $I-\hat{P}$, where $I$ stands for  the identity matrix. Furthermore,
\begin{equation}
\label{eq:pi=A}
\forall x\in \Delta, \ \pi(x) = \frac{x A}{\langle  x A, \mathbf{1}\rangle}.
\end{equation}
Indeed, let $\gamma = \sum_{i \in F^*} \pi(x)(i) P_{i,0}.$  Then
$$
\pi(x) K[x] = \pi(x) \Leftrightarrow \pi(x) \cdot (\hat{P}-I) = - \gamma x \Leftrightarrow  \pi(x)  =   \gamma x \cdot (I-\hat{P})^{-1} = \gamma x \cdot A,
$$
and as $\pi(x) \in \Delta$, we have
$$
1=\sum_{i \in F^*} \pi(x)(i) = \gamma \sum_{i \in F^*}  (x \cdot A)(i)= \gamma \langle  x A, \mathbf{1}\rangle.
$$

The next lemma follows from classical results on linear dynamical systems.
\begin{lem}[Long time behavior of $\Phi$]
\label{lem:cvexpo}
For all $\alpha \in (0, R)$, there exists $C=C_\alpha>0$  such that for all $x\in \Delta$ and $t\geq 0$, we have
\begin{equation}
\label{eq:borne-expo}
 \Vert \Phi(t,x) -  \nu \Vert \leq C e^{-\alpha t} \Vert \Phi( 0,x) -  \nu \Vert.
\end{equation}
\end{lem}
\begin{proof}
Let us consider $\Phi_1:(t,x)\mapsto x \cdot e^{tA}$. Writing $x= \nu + (x-\nu)$ and using $\nu A =(1-\lambda)^{-1}\nu$, it comes
\begin{equation}
\label{eq:flow-lin}
\Phi_1(t,x) =  e^{ (1-\lambda)^{-1} t} \left(\nu  + (x-\nu) e^{t(A- (1-\lambda)^{-1} I)}\right).
\end{equation}
Let
$$
\beta < (1-\lambda)^{-1} -\max_{i\geq 2} \mathrm{Re}((1-\lambda_i)^{-1}).
$$
 Using for instance a Dunford decomposition, we get that
for $t$ large enough
$$
\Vert e^{t(A- (1-\lambda)^{-1} I)} \Vert = \sup_{\Vert u \Vert =1} \Vert e^{t(A- (1-\lambda)^{-1} I)} u \Vert \leq e^{-\beta t}
$$
Let now $\Phi_2$ be the semiflow on $\Delta$ defined  for all $t \geq 0$ and $x \in \Delta$ by
$$
  \Phi_2(t,x) = \frac{\Phi_1(t,x) }{\langle\Phi_1(t,x), \mathbf{1} \rangle }.
$$
For every $t\geq 0, x\in \Delta$, $\Phi_2(t,x)$ belongs trivially to $\Delta$ because $xe^{tA}$ possesses positive coordinates and $\langle \Phi_2(t,x), \mathbf{1} \rangle=1$.
It follows from \eqref{eq:flow-lin} that for some $C>0$,
$$
\forall t\geq 0, \ \Vert \Phi_2(t,x) - \nu \Vert \leq C e^{-\beta t} \Vert x  - \nu \Vert.
$$
Now, note that  $\Phi_2$ and $\Phi$  have the same orbits (up to a time re-parametrization). Indeed, differentiating in $t$, we find that
\begin{equation*}
  \left\{
      \begin{aligned}
        \forall t\geq 0, \forall x\in \Delta &, \ \partial_t \Phi_2(t,x) =  \langle  \Phi_2(t,x) A, \mathbf{1} \rangle \left(\frac{\Phi_2(t,x) A }{\langle \Phi_2(t,x) A, \mathbf{1} \rangle}  - \Phi_2(t,x)\right)   , \\
        \forall x\in \Delta &, \ \Phi_2(0,x)=x.
      \end{aligned}
    \right.
\end{equation*}
Hence,
\begin{equation}
\label{eq:phi=phi2}
\forall t\geq 0, \ \forall x\in \Delta, \ \Phi(s(t,x),x) = \Phi_2(t,x),
\end{equation}
where
$$
s(t,x)= \int_0^t  \langle \Phi_2(x,u) A , \mathbf{1} \rangle  du.
$$
This mapping  (i.e. $t\mapsto s(t,x)$) is strictly increasing because $\Phi_2(x,u)$ belongs to  $\Delta$ so that $\langle  \Phi_2(x,u)A, \mathbf{1} \rangle>0$  for all $u\geq 0$. It follows from \eqref{eq:flow-lin} that $s(t,x)/t$ tends to $(1-\lambda)^{-1}$, uniformly in $x\in \Delta$ as $t$ tends to infinity. Thus, fixing $\alpha <\beta(1- \lambda)< R$, for $t$ large enough, we have $\beta t > \alpha s(t,x)$ and, consequently,
$$
\Vert \Phi(s(t,x),x) - \nu \Vert \leq C e^{-\alpha s(t,x)} \Vert x- \nu \Vert \Leftrightarrow \Vert \Phi(s,x) - \nu \Vert \leq C e^{-\alpha s} \Vert x- \nu\Vert,
$$
for $s$ large enough. Replacing $C$ by a sufficiently larger constant, the previous inequality holds for all time and this proves the Lemma.
\end{proof}
\begin{Rq}[Probabilist interpretation of $A,\Phi_1,\Phi_2$]
Observe that $A$ is the Green function defined as
$$
\forall i,j \in F^*, \ A_{i,j} = \E_i \left[\sum_{k\geq 0} \un_{Y_k = j} \right].
$$ In particular,
$$
(A\mathbf{1})_i =\E_i\left[T_0\right],$$
where $T_0=\inf\{n\geq 0 \ | \ Y_n=0 \}$. Moreover, $\Phi_1$ can be understood as the main measure of a branching particle system and  $\Phi_2$ is then the renormalized main measure. See \cite{C11} or \cite[Chapitre 4]{C-these13} for details.
\end{Rq}

\begin{coro}[Gradient estimate]
\label{cor:gradient}
Let  $Dh(\nu)$ denote the Jacobian matrix of $h$ at $\nu.$ Then
$$\lim_{t \rightarrow \infty} \frac{1}{t} \ln(\|e^{tDh(\nu)}\|) = \limsup_{t \rightarrow \infty} \frac{1}{t} \ln(\|e^{tDh(\nu)}\|) \leq -R.$$ In particular, eigenvalues of $Dh(\nu)$ have their real parts bounded by $-R$.
\end{coro}
\begin{proof}
Set $\Phi_t(\cdot) =\Phi(t,\cdot).$ The mapping $x\mapsto h(x)$ being  $C^\infty,$  classical results on ordinary differential equations imply that $\Phi$ is $C^\infty$ and satisfies the variational equation
$$
\partial_t D\Phi_t(x) = Dh(\Phi_t(x)) \cdot D\Phi_t(x)
$$
with initial condition $D\Phi_0(x) = I.$
Thus, because $\Phi_t(\nu)= \nu,$ $$D\Phi_t(\nu) = e^{t D h(\nu)}$$
Now, fix $t\geq 0$ and $\alpha< R.$  On the first hand, using Lemma \ref{lem:cvexpo} we get that
$$
s^{-1} \Vert \Phi_t(\nu+ su) -\Phi_t(\nu)\Vert \leq C e^{-\alpha t} \Vert u\Vert
$$
for every $s\geq 0$ and $u\in \R^d.$ Taking the limit $s\rightarrow 0$ leads to
$$
\Vert e^{t D h(\nu)} . u \Vert = \Vert D\Phi_t(\nu) \cdot u \Vert \leq C e^{-\alpha t} \Vert u\Vert.
$$
This ends the proof.
\end{proof}

\subsection{Links between $(x_n)_{n\geq 0}$ and $\Phi$}
\label{sect:synthese}
Let us rapidly recall some definitions of \cite{B99}. To this end, we define the following continuous time interpolations $\widehat{X}, \bar{X}, \bar{\epsilon}, \bar{\gamma} : \R_+ \to \R^d$ by
$$
\widehat{X}(\tau_n +s)= x_n + s\frac{x_{n+1}-x_n}{\tau_{n+1}-\tau_n}, \ \bar{X}(\tau_n +s)=x_n, \ \bar{\epsilon}(\tau_n +s)=  \epsilon_{n+1} \ \text{ and } \ \bar{\gamma}(\tau_n +s)= \gamma_{n+1},
$$
for every $n\in \N$ and $s\in [0,\gamma_{n+1})$.  We also set $m: t \mapsto \sup\{k \geq 0  \ | \ t\geq \tau_k \}$. A continuous map $Z: \R_+ \mapsto \Delta$ is called an \textit{asymptotic pseudo-trajectory} of $\Phi$ if for all $T>0$,
$$
\lim_{t\rightarrow + \infty} \sup_{0 \leq s \leq T }\Vert Z(t+s) - \Phi_{s} (Z_t) \Vert=0.
$$
Given $r<0$, it is called a $r-$pseudo-trajectory of $\Phi$ if
$$
\limsup_{t \rightarrow + \infty} \frac{1}{t} \ln\left( \sup_{0 \leq s \leq T }\Vert Z(t+s) - \Phi_{s} (Z_t) \Vert\right)\leq r,
$$
for some (or all) $T>0$. We have

\begin{lem}[Pseudo-trajectory property of $\widehat{X}$]
\label{lem:APT}
With probability one, $\widehat{X}$ is an asymptotic pseudo-trajectory of $\Phi$. If furthermore $l(\gamma)<0$ then $\widehat{X}$ is almost surely a $l(\gamma)/2$-pseudo-trajectory of $\Phi$.
\end{lem}
\begin{proof}

The proof is similar to \cite[Section 5]{B97}. 
Let $\boldsymbol{\pi}(x)$ be the matrix over $F^*$ be defined by   $\boldsymbol{\pi}(x)_{i,j}= \pi(x)_j.$

By irreducibility of $K[x]$ the continuous time Markov semi-group $(e^{t (K[x]-I)})_{t\geq 0}$ converge at an exponential rate toward  $\boldsymbol{\pi}(x).$ Thus,  for all $x\in \Delta$  the matrix
$$
Q[x] = - \int_0^\infty ((e^{t (K[x]-I)}) - \boldsymbol{\pi}(x)) dt
$$ is well defined.
Using that $(K[x]-I)$ is the generator of the semigroup $(e^{t (K[x]-I)})_{t\geq 0}$, it is classic (and easy) to see that $Q[x]$ is  solution to the Poisson equation:
$$
(I-K[x])Q[x]=Q[x] (I-K[x]) = I -  \boldsymbol{\pi}(x).
$$
We can write
$$
\gamma_n \epsilon_n=\delta_n^1 + \delta_n^2 + \delta_n^3 + \delta_n^4,
$$
where, for all $j\in F^*$, we have
$$
\delta_n^1(j) = \gamma_n \left( Q[x_n]_{X_{n+1},j} - K[x_n] Q[x_n]_{X_{n},j}\right),
$$
$$
\delta_n^2(j) = \gamma_n K[x_n] Q[x_n]_{X_{n},j} - \gamma_{n-1} K[x_n] Q[x_n]_{X_{n},j},
$$
$$
\delta_n^3(j) = \gamma_{n-1} K[x_n] Q[x_n]_{X_{n},j} - \gamma_n K[x_{n+1}] Q[x_{n+1}]_{X_{n+1},j},
$$
and
$$
\delta_n^4(j) = \gamma_n \left( K[x_{n+1}] Q[x_{n+1}]_{X_{n+1},j} - K[x_{n}] Q[x_{n}]_{X_{n+1},j}\right).
$$
Continuity, smoothness of $Q$ and compactness of $\Delta$ ensure the existence of $C>0$ such that
$$
\Vert \delta^2_n \Vert \leq C (\gamma_{n-1} - \gamma_n ), \ \Vert \sum_{i=n}^k \delta^3_{i} \Vert \leq C \gamma_{n-1} \ \text{ and } \ \Vert \delta^4_n \Vert \leq C\gamma_n \Vert x_{n+1} - x_n \Vert \leq C \gamma_n^2.
$$
Now, if $\mathcal{F}_n = \sigma\{ X_k \ | \ k\leq n \}$, the term $\delta^1_n$ is a $\mathcal{F}_{n-1}$-martingale increment and there exists $C_1>0$ such that $\Vert \delta^1_n \Vert^2 \leq C_1\gamma_n^2$. Namely, for every $n_0\geq 0$ , if for all $n\geq n_0+1$, $M_n = \sum_{k=n_0}^{n-1} \delta^1_k$ then $(M_n)_{n\geq n_0+1}$ is a $(\mathcal{F}_{n})_{n\geq n_0+1}$-martingale.
From these inequalities, we have
\begin{align*}
\Delta(t,T)
&=\sup_{0 \leq u \leq T} \Vert \int_t^{t+u} \bar{\epsilon}(s) ds \Vert \\
&\leq \sup_{0 \leq i \leq T} \Vert \int_{\tau_{m(t)}}^{\tau_{m(t+u)}} \bar{\epsilon}(s) ds \Vert  + \sup_{0 \leq u \leq T} \Vert \int_{\tau_{m(t)}}^{t} \bar{\epsilon}(s) ds \Vert + \sup_{0 \leq u \leq T} \Vert \int_{\tau_{m(t+u)}}^{t+u} \bar{\epsilon}(s) ds \Vert \\
&\leq \sup_{0 \leq u \leq T} \Vert \sum_{j=m(t)}^{m(t+u)-1}  \gamma_{j+1} \epsilon_{j+1} \Vert +C_2 \sup_{0 \leq u \leq T} |t- \tau_{m(t)}| + C_2 \sup_{0 \leq u \leq T} |t+u- \tau_{m(t+u)}| \\
&\leq \sup_{0 \leq u \leq T} \Vert \sum_{j=m(t)}^{m(t+u)-1} \delta^1_{j+1} \Vert + \sup_{0 \leq u \leq T} \Vert \sum_{j=m(t)}^{m(t+u)-1} \delta^2_{j+1} \Vert + \sup_{0 \leq u \leq T} \Vert \sum_{j=m(t)}^{m(t+u)-1}  \delta^3_{j+1} \Vert \\
&\quad  \ \quad \ + \sup_{0 \leq u \leq T} \Vert \sum_{j=m(t)}^{m(t+u)-1} \delta^4_{j+1} \Vert + 2 C_2 \gamma_{m(t)+1} \\
&\leq \sup_{0 \leq u \leq T} \Vert \sum_{j=m(t)}^{m(t+k)-1}  \delta^1_{j+1} \Vert + C \gamma_{m(t)}  + C \gamma_{m(t)} +  C_3 T \gamma_{m(t)+1} + C_4 \gamma_{m(t)+1},\\
\end{align*}
for some $C_3,C_4>0$. If we set $U_{n+1}=\delta^1_n$ then following \cite[Proposition 4.4]{B99}, under \eqref{eq:hyp-gamma}, we see that the last term tends to zero. Using \cite[Proposition 4.1]{B99}, this proves the first part of the statement. Let us now prove that it is a $l(\gamma)/2-$pseudo-trajectory. Thanks to \eqref{eq:hyp-gamma}, Inequality $(11)$ of \cite[Proposition 4.1]{B99} and the beginning of the proof of \cite[Proposition 8.3]{B99}, it is enough to prove that $\limsup_{t\rightarrow \infty} \ln(\Delta(t,T))/t \leq l(\gamma)/2$. From the previous decomposition, it is enough to prove
$$\limsup_{t\rightarrow \infty} \frac{1}{t}\ln\left(\sup_{0 \leq k \leq T} \Vert \sum_{j=m(t)}^{m(t+k)-1}  \delta^1_{j+1} \Vert\right) \leq l(\gamma)/2 \ \text{a.s.}$$ and
again the end of the proof is the same as in the Robbins-Monro algorithm situation (see the proof of \cite[Proposition 8.3]{B99}).
\end{proof}

\subsection{Proof of the main results}
\label{sect:main-proof}

\begin{proof}[Proof of Theorem \ref{th:main}]
By Lemma \ref{lem:cvexpo}, $\{ \nu\}$ is a global attractor for $\Phi$; a global attractor is an attractor whose basin is all the space, see \cite[page 22]{B99}. Thus, it contains the limit set of every (bounded) asymptotic pseudo-trajectory (see e.g \cite[Theorem 6.9]{B99} or \cite[Theorem 6.10]{B99}). Lemma \ref{lem:APT} gives the almost-sure convergence. The second part of Theorem \ref{th:main} follows directly from  \cite[Lemma 8.7]{B99} and Lemma \ref{lem:APT}.
\end{proof}

\begin{proof}[Proof of corollary \ref{cor:Aldous}]
Since the limsup in the definition of $l(\gamma)$ is a limit, the result is a direct consequence of Theorem \ref{th:main}.
\end{proof}

\begin{proof}[Proof of Theorem \ref{th:tcl}]
Let us check that our model satisfies the assumptions of \cite[Theorem 2.1]{F13}.  Hypothesis $\mathbf{C1}$ holds because of Perron-Frobenius Theorem for $(a)$, Corollary \ref{cor:gradient} for $(c)$ ($(b)$ is trivial). Using the notations of this paper and the one of the proof of Lemma \ref{lem:APT}, we have
$$
e_{n+1} = \gamma_n^{-1} \delta^1_n \ \text{ and } \ r_n = \gamma_n^{-1} (\delta^2_n + \delta^3_n + \delta^4_n).
$$
Assumption $\mathbf{C2} (a)$ holds, Assumption $\mathbf{C2} (b)$ holds with $\mathcal{A}_m= \mathcal{A}_{m,k}= \Omega$, where $\Omega$ is our probability space. Note that $x_n\rightarrow \nu$ with probability one.

Assumption $\mathbf{C2} (c)$ is more tricky but usual. Indeed, let us mimic \cite[Section 4]{F13}. Let
$$
F_x(X)_{i,j}=  (\sum_{k \in F^*} Q[x]_{k,j} Q[x]_{k,i} K[x]_{X,k}) -K[x]Q[x]_{X,i} K[x]Q[x]_{X,j}
$$
be a kind of covariance matrix. We have
\begin{align*}
\E\left[ e_{n+1}(i) e_{n+1}(j) \ | \ \mathcal{F}_n \right]
&= F_{x_n}(X_n)_{i,j}= (U_*)_{i,j} + (D^1_n)_{i,j} + (D^2_n)_{i,j},
\end{align*}
where $U_* = \sum_{k \in F^*} F_{\nu}(k) \nu_k$,
$$
D^1_n= \sum_{k \in F^*} (F_{x_n}(k) \pi(x_n)_k - F_{\nu}(k) \nu_k)
$$
which tends almost surely to $0$ thanks to Theorem \ref{th:main} and
$$
D^2_n= F_{x_n}(X_n) -\sum_{k \in F^*} F_{x_n}(k) \pi(x_n)_k.
$$
It rests to prove that
$$
\lim_{n \rightarrow \infty} \gamma_n \E\left[\Vert \sum_{m=1}^n D^2_m \Vert \right] = \lim_{n \rightarrow \infty} \gamma_n \E\left[\Vert \sum_{m=1}^n \left( F_{x_m}(X_m) -\sum_{k \in F^*} F_{x_m}(k) \pi(x_m)_k \right) \Vert \right] =0.
$$
To this end, we use again the solution of the Poisson equation $Q$ introduced in the proof of Lemma \ref{lem:APT}. Indeed following \cite{F13} we set $U_x(X)=Q[x] F_x(X)$, which satisfies
$$
(I-K[x])U_x = F_x(X) -\sum_{k \in F^*}  F_x(k) \pi(x)_k,
$$
and
$$
D^{2,a}_n= U_{x_n}(X_{n+1}) - K[x_n] U_{x_n}(X_n),
$$
$$
D^{2,b}_n= U_{x_n}(X_n) - U_{x_n}(X_{n+1}).
$$
We have $D^2_n=D^{2,a}_n + D^{2,b}_n$. Arguments that follow come from directly to \cite[page 16]{F13}. Note that, with the notations of \cite{F13}, Assumption A3 holds with constant  functions $V_1$ and $V_2$ , $b=1$, $\bar{\tau}=1$. Indeed recall that the state space is finite and then all regularity and boundedness assumptions are satisfied. From Theorem \ref{th:main}, the convergence assumptions also hold.\\
 Sequence $(D^{2,a}_n)_{n\geq 0}$ is a martingale increment sequence and, thanks to Burkholder inequality, we have
$$
\E\left[\Vert \sum_{m=1}^n D^{2,a}_m \Vert^{2} \right] \leq C n,
$$
for some $C>0$. Now Cauchy-Schwarz inequality gives
$$
\lim_{n \rightarrow \infty} \gamma_n \E\left[\Vert \sum_{m=1}^n D^{2,a}_m \Vert \right] \leq \lim_{n \rightarrow \infty} \gamma_n \sqrt{C n}=0,
$$
because $\sum_{n\geq 0} \gamma_n=+\infty$ and $\sum_{n\geq 0} 1/\sqrt{n}<+\infty$.
\\
Sequence $(D^{2,b}_n)_{n\geq 0}$ can be written as a telescopic sum and then
$$
\E\left[\Vert \sum_{m=1}^n D^{2,a}_m \Vert^2 \right] \leq C\left( 1+ \sum_{k=1}^n \gamma_k^2 \right),
$$
for some $C>0$ (again see \cite[page 16]{F13} for details). The limiting assumption is then satisfied because $\sum_{n\geq 0} \gamma_n^2<+\infty$.
 Assumption $\mathbf{C3}$ is then satisfied. Finally, the last assumption $\mathbf{C4}$ is supposed to be true in our setting.
\end{proof}

\begin{Rq}[SA with controlled Markov chain
dynamics]
Our sequence $(x_n)_{n\geq 0}$ is an instance of the so-called \textit{SA with controlled Markov chain
dynamics} introduced in \cite[Section 4]{F13}. Instead of our proof for the central limit theorem, we could use \cite[Proposition 4.1]{F13}. Nevertheless, it would weaken the assumption on $(\gamma_{n})_{n\geq 0}$ (see Assumption \textbf{A2}).
\end{Rq}

\begin{proof}[Proof of Corollary \ref{cor:tcl}]
The $L^p-$norm are continuous bounded functions on $\Delta$ thus the result is straightforward.
\end{proof}

\begin{proof}[Proof of Corollary \ref{cor:Xn}]
By irreducibility of $P$ (and hence $K[\nu]$), $\nu_i > 0$ for all $i.$ Thus,  $K[\nu]_{ii} \geq P_{i 0} \nu_i > 0$ for all $i$ such that $P_{i 0} > 0.$ This shows that  $K[\nu]$ is aperiodic. Therefore, by the ergodic theorem for finite Markov chains, there exist $C_0>0$ and $\rho \in [0,1)$ such that for all $x \in \Delta$
$$
\Vert x K^n[\nu] - \nu \Vert \leq C_0 \rho^n.
$$
In particular, $\nu$ is a global attractor for the discrete time dynamical system on $\Delta$ induced by the  map $x  \mapsto x K[\nu].$ To prove that $\mu_n \rightarrow \nu$ it then suffices to prove that $(\mu_n)$ is an asymptotic pseudo trajectory of this dynamics (that is  $\|\mu_n K[\nu] - \mu_{n+1}\| \rightarrow 0$) because the limit set of a bounded asymptotic pseudo-trajectory is contained in every global attractor (see e.g  \cite[Theorem 6.9]{B99} or \cite[Theorem 6.10]{B99}). Now,
\begin{align*}
&\Vert  \mu_n K[\nu] - \mu_{n+1}  \Vert
= \sum_{j\in F^*} \left| \mu_{n} K[\nu](j) - \mu_{n+1}(j) \right|
 = \sum_{j\in F^*} \left| \E\left[K[\nu]_{X_n,j} - K(x_n)_{X_n,j}\right] \right|\\
& = \sum_{j\in F^*} \left| \E\left[ P_{X_n,0} (\nu(j)-x_n(j)) \right] \right|  \leq \max_{i \in F^*} P_{i,0} \E\left[\Vert \nu -x_{n} \Vert\right]
\end{align*}
and the proof follows from Theorem \ref{th:main} and dominated convergence.

If one now suppose that assumptions of Corollary \ref{cor:tcl} hold, then, in view of the preceding inequality, there exists $C>0$ such that $$\Vert  \mu_n K[\nu] - \mu_{n+1}  \Vert \leq C \sqrt{\gamma_n}.$$
Therefore $$
\|\mu_{n+p} - \mu_n K[\nu]^p\| = \| \sum_{i=0}^{p-1} \left(\mu_{n+i} K[\nu] - \mu_{n+i+1}\right) K[\nu]^{p+i-1} \| \leq  C \sum_{i = 0}^{p-1} \sqrt{\gamma_{n+i}} \leq p C \sqrt{\gamma_{n}}$$ and
$$\| \mu_{n+p} - \nu \| \leq \|\mu_{n+p} - \mu_n K[\nu]^p\| + \| \mu_n K[\nu]^p - \nu \| \leq p C \sqrt{\gamma_{n}} +  C_0 \rho^p.$$

\end{proof}

\section{A second model based on interacting particles}
\label{sect:FV}
In continuous-time, a mainstream method to simulate QSD is the so-called Fleming-Viot particle system. It was introduced and well studied in \cite{BHM00} for the Brownian motion and in \cite{MM00} for general Markov processes. See also \cite{CT13,FM07,GJ12,V11}. Here, we study and give some applications of \cite{BW03} for a discrete version of this algorithm.
This one is based on a particle system evolving as follow: at each time, we choose, uniformly at random, a particle $i$ and replace it by another one $j$; this one is choosen following the probability $P_{i,j}$ or uniformly on the others particles with probability $P_{i,0}$. In this work we will study a slight modification; we allow us the choice to replace the died particle on its previous position. More precisely, let $ N\geq 2$ and consider $(X^N_n)_{n\geq 0}$ be the Markov chain on $\Delta$ with transition
\begin{equation}
\label{eq:mean-field}
\p\left( X^N(n+1) = x + \frac{1}{N}(\delta_j - \delta_i)  \ | \ X^N(n)=x \right) = p_{i,j}(x),
\end{equation}
where
\begin{equation}
\label{eq:p=K}
p_{i,j}(x) = x(i) \left(P_{i,j} + P_{i,0} x(j)\right) = x(i) K[x]_{i,j},
\end{equation}
for every $x\in \Delta$, $n \geq 0$, $i,j \in F^*$. This algorithm is relatively close to the one used in \cite{DM13,MG99} in non-linear filtering. In their setting, all particles move and die at each step. In our setting, only one particle moves at each step, and this dynamics is then closer to the continuous-time algorithm. We are interested in the limit of Markov chains $X^N$, when $N$ is large, and with the time scale $\delta = 1/N$. The key element for such approximation
is the vector field $F=(F_j)_{j\in F^*}$, defined by
$$
\forall x\in \Delta, \forall j\in F^*, \ F_j (x)= \sum_{i \neq j } (p_{i,j}(x) - p_{j,i}(x)),
$$
which, for large $N$ and short time intervals, gives the expected net increase share during the time interval, per time unit. The associated mean-field flow $\Psi$ is the solution to
\begin{equation}
\label{eq:psi}
  \left\{
      \begin{aligned}
        \forall t\geq 0, \forall x\in \Delta &, \ \partial_t \Psi(t,x) = F(\Psi(t,x)), \\
        \forall x\in \Delta &, \ \Psi(0,x)=x.
      \end{aligned}
    \right.
\end{equation}
Using \eqref{eq:p=K}, we have
$$
\forall j \in F^*, \forall x\in \Delta, \  F_j(x)= \sum_{i \in F^*} x_i (P_{i,j} + x_j P_{i,0}) -x_j,
$$
and $\Psi$ is then the conditioned semi-group of the absorbed Markov process $(U_t)_{t\geq 0}$ generated by $(P-I)$. More precisely, for all $j \in F^* $, $t\geq 0$ and $x\in \Delta$, we have
$$
\Psi(t,x) = \frac{\sum_{i \in F^*} x(i) \p\left( U_t =j \ | \ U_{0} =i \right)}{\sum_{i \in F^*} x(i)  \p\left( U_t \neq 0 \ | \ U_{0}=i \right)} = \frac{x e^{t(P-I)}}{\langle x e^{t(P-I)}, \mathbf{1} \rangle}.
$$

This model was studied in a more general setting in \cite{BW03}. In particular if we set
$$
\forall s\in [0,1), \ \bar{X}^N ((n+s)/N )= X^N_n + s (X^N_{n+1}-X^N_n),
$$
then we have

\begin{theo}[Deviation inequality]
 For every $T>0$, there exists a (explicit) constant $c=c_T>0$ such that for any $\varepsilon>0$, $x\in \Delta$ and $N$ large enough,
$$
\p\left( \max_{0 \leq t \leq T} \Vert \bar{X}^N(t) - \Psi(t,x) \Vert \geq \varepsilon  \ | \ X^N(0)= x \right) \leq 2 d e^{-c \varepsilon^2 N}.
$$
In particular,  for all $\theta<1/2$, we have
\begin{equation}
\label{eq:vitesseps}
\lim_{N \rightarrow + \infty } N^{\theta} \max_{0 \leq t \leq T} \Vert \bar{X}^N(t) - \Psi(t,x) \Vert = 0 \ \text{ a.s.}
\end{equation}
and
\begin{equation}
\label{eq:qsdps}
\lim_{N \rightarrow + \infty} \lim_{ n \rightarrow + \infty} X^N_n = \nu \ \text{ a.s.}
\end{equation}
\end{theo}
\begin{proof}
It comes from  \cite[Lemma 1]{BW03}, Borel-Cantelli Lemma and  \cite[Proposition 6]{BW03}. The constant $c$ is given by
$$
c= \frac{e^{-2 l_F T}}{8T\sqrt{\sqrt{2} + \Vert F \Vert_2^2}},
$$
where $l_F$ is the Lipschitz constant of $F$ on the compact set $\Delta$ and $\Vert F \Vert_2^2$ the supremum of $\Vert F(x) \Vert_2^2$ over $\Delta$.
\end{proof}

\begin{Rq}[Continuous-time case]
Firstly, if we consider our discrete-time algorithm indexed by a Poisson process, we recover the Fleming-Viot algorithm, see \cite[Section 6]{BW03} for details. This enables us to compare this result with previous works on Fleming-Viot algorithm. Articles \cite[Theorem 1]{V11} and \cite[Theorem 1.1]{MM00} give a $L^1$-bound in a more general setting (not finite state space) but to our knowledge, \eqref{eq:vitesseps} and \eqref{eq:qsdps} are the first almost-sure convergence results.\\
However, none of these works give a rate of convergence to the QSD.  Using $t=\gamma \ln(N)$ in \cite[Corollary 1.5]{CT13} (and its proof) and \cite[Remark 2.8]{CT13}, we have a uniform error term in $N^{-\gamma}$ for the approximation of the QSD, where $\gamma$ depends on the rate of convergence of the conditioned semi-group to equilibrium.  Even if our setting is in discrete time, this result can be compared with our Theorem \ref{th:main} (and Corollary \ref{cor:tcl}, more precisely).
\end{Rq}

\begin{Rq}[Time versus spatial empirical measure]
In this work, we compare two dynamics based on $K[\mu_r]$ where $\mu_r$ is either the time occupation measure or the spatial occupation measure. The analysis of the resultant flows, $\Phi$ and $\Psi$ are very similar. This analogy was already observed in others works with the Mc Kean-Vlasov equation; see \cite{B14,BLR02}.
\end{Rq}





\textbf{Acknowledgement} We acknowledge financial support from the Swiss National Foundation Grant  FN 200020-149871/1.
The work of B. Cloez was partially supported by the CIMI (Centre International de Math\'{e}matiques et d'Informatique), ANR-11-LABX-0040-CIMI within the program ANR-11-IDEX-0002-02
.


\end{document}